\documentclass[12pt]{iopart}

\usepackage{graphicx,setspace}
\usepackage{psfrag}
\usepackage{amsfonts}
\usepackage{amssymb, amsthm}
\usepackage{mathrsfs}
\usepackage{epstopdf}
\usepackage{float}
\usepackage{lineno,hyperref}
\usepackage{color}
\usepackage{subfigure}
\usepackage{cases}

\usepackage{fancyhdr}

\modulolinenumbers[5]

\begin{document}

\newtheorem{definition}{Definition}
\newtheorem{lemma}{Lemma}
\newtheorem{remark}{Remark}
\newtheorem{theorem}{Theorem}
\newtheorem{proposition}{Proposition}
\newtheorem{assumption}{Assumption}
\newtheorem{example}{Example}
\newtheorem{corollary}{Corollary}
\newtheorem{conclusion}{Conclusion}
\def\ep{\varepsilon}
\def\Rn{\mathbb{R}^{n}}
\def\Rm{\mathbb{R}^{m}}
\def\E{\mathbb{E}}
\def\hte{\hat\theta}

\renewcommand{\theequation}{\thesection.\arabic{equation}}

\title{ Data assimilation and parameter estimation for a multiscale stochastic system  with $\alpha$-stable L\'evy noise}
\author{Yanjie Zhang$^1$,
Zhuan Cheng$^2$,
Xinyong Zhang$^3$,
Xiaoli Chen$^1$,
Jinqiao Duan$^2$ and Xiaofan Li$^2$
}

\address{$^1$Center for Mathematical Sciences \& School of Mathematics and Statistics, Huazhong University of Science and Technology, Wuhan 430074,  China}
\address{$^2$Department of Applied Mathematics, Illinois Institute of Technology, Chicago, IL 60616, USA}
\address{$^3$Department of Mathematical Sciences, Tsinghua University, Beijing 100084,  China}

\ead{zhangyanjie2011@163.com, zcheng9@hawk.iit.edu, zhangxinyong12@mails.tsinghua.edu.cn,
xlmath804@163.com, duan@iit.edu and lix@iit.edu  .}

\address{$^*$This work was partly supported by the NSF grant 1620449,  and NSFC grants 11531006, 11371367,  and 11271290.}

\begin{abstract}
This work is about low dimensional reduction for a slow-fast data assimilation system with non-Gaussian $\alpha-$stable L\'evy noise via stochastic averaging. When the observations are only available for slow components, we show that the averaged, low dimensional filter approximates the original filter, by examining the corresponding Zakai stochastic partial differential equations. Furthermore, we demonstrate that the low dimensional slow system approximates the slow dynamics of the original system, by examining   parameter estimation and most probable paths.

\end{abstract}

\section{Introduction}
Data assimilation is a procedure to extract system state information with the help of observations \cite{Crisan}. The state evolution and the observations are usually under random fluctuations. The general idea  is to gain the best estimate for the true system state, in terms of the probability distribution for the system state, given only some noisy observations of the system. It provides a recursive algorithm for estimating a signal or state of a random dynamical system based on noisy measurements. It is also very important in many practical applications from inertial guidance of aircrafts and
spacecrafts to weather and climate prediction. Most of the existing works on data assimilation is conducted in the context of Gaussian random fluctuations.
The effects of multiscale signal and observation processes in the context of Gaussian random fluctuations has been considered by Park et.al.(see \cite{Park}), and they have shown that the probability density of the original system converges to that of the reduced system, by a Fourier analysis method. Imkeller et.al.\cite{Imkeller} have further proved the convergence in distribution for the optimal filter, via backward stochastic differential equations and asymptotic techniques.

However, random fluctuations are often non-Gaussian (in particular, L\'evy type) in nonlinear systems, for example,  in geosciences \cite{Dit}, and biosciences \cite{ Lin, Jia, Roberts, Bli, Al, Cg, Zheng}. There are experimental demonstrations of L\'evy fluctuations  in optimal foraging theory, rapid geographical spread of emergent infectious disease and switching currents. Humphries {\it et. al.} \cite{Humphries} used GPS to track the wandering black bowed albatrosses around an island in Southern Indian Ocean to study the movement patterns of searching food. They found that by fitting the data of the movement steps, the movement patterns obeys the power-law property with power parameter $\alpha=1.25$.   La Cognata {\it et. al.} \cite{Al} considered a Lotka-Volterra system of two competing species subject to multiplicative $\alpha$-stable L\'evy noise, and   analyzed the role of the L\'evy noise sources.   Lisowski {\it et. al.} \cite{Bli} studied  a model of a stepping molecular motor consisting of two connected heads, and examined its dynamics  as noise parameters varied. Speed and direction appear to very sensitively depend on characteristics of the noise.
They explored the effect of noise on the ballistic graphene-based small Josephson junctions in the framework of the resistively and capacitively shunted model and found that the analysis of the switching current distribution made it possible to efficiently detect a non-Gaussian noise
component in a Gaussian background.

L\'evy motions are  appropriate models for a class of important non-Gaussian processes with jumps or bursts \cite{Lin, Jia, jd}. It is desirable to  consider  data assimilation when the system evolution is under  L\'evy motions. This has been recently considered by one of us and other authors but not in the multiscale context (see \cite{cckc2, spss, q2, sun}).

The multi-scale stochastic dynamical systems arise widely in finance and biology. For example, there are two kinds of mutual securities in financial markets.
One for the low-risk bonds, which can be characterized by ordinary differential equations; the other for high-risk stocks, whose price has two different changes. On the one hand, the edge change caused by the normal supply and demand can be characterized by Gaussian noise.
On the other hand, due to the arrival of the important information of the stock, there will be a finite jump in the stock price. Such perturbations can be characterized by non-Gauss noise. In general, stock prices change at all times, while bond prices change for months or even years. Thus, the price of these two securities can be characterized by two-scales system with non-Gaussian
noise (see \cite{gid}). Moreover, a large number of observations from biological  experiments showed that the production of mRNA and proteins occur in a bursty, unpredictable, and intermittent manner, which create variation or "noise" in individual cells or cell-to-cell interactions. Such burst-like events appear to be appropriately modeled by the non-Gaussian noise. Since the mRNA synthesis process is faster than the protein dynamics, this leads to a two-time-scale system (see \cite{wf}). Here $\epsilon $ represents the ratio between the natural time scales of the protein and mRNA.

Parameter estimation for continuous time stochastic models is an increasingly
important part of the overall modeling strategy in a wide variety of applications.
It is quite often the case that the data to be fitted to a diffusion process has a
multiscale character with Gaussian noise. One example is in molecular dynamics, where
it is desirable to find effective models for low dimensional phenomena (such as
conformational dynamics, vacancy diffusion and so forth) which are embedded
within higher dimensional time-series. We are often interested in the parameter (see \cite{wf}),  which represents the degradation or production rates of protein and mRNA. In this paper, we develop a parameter estimation method for multiscale diffusions with non-Gaussian noise. The results established here may be used to examine the change rate for low-risk bounds (see \cite{gid}).

In this present paper, we consider a \emph{slow-fast} data assimilation system   under  L\'evy noise, but only the slow component $ X_t $ is observable. By averaging out the fast component via an invariant measure, we thus reduce the system dimension by focusing on the slow system evolution. We prove that the reduced lower dimensional filter effectively approximates (in probability distribution) the original filter. We  demonstrate that a system parameter may be estimated via the low dimensional slow system, utilising only observations on the slow component. The accuracy for this estimation is quantified by $p$-moment, with $p\in(1,\alpha)$. We apply the stochastic Nelder-Mead method \cite{ck} for optimization in the searching for the estimated parameter value. Furthermore, we  illustrate the low dimensional approximation by  comparing  the most probable paths for the slow system and the original system. Finally, we make some remarks in Section 6.

This paper is organized as follows.  After recalling some basic facts about L\'evy motions and the generalized solution in the next section,  we address the effects of the multiscale signal and observation processes in the context of L\'evy random fluctuations in Section 3 to Section 5.  We illustrate  the low dimensional slow approximation by examining zakai equation, parameter estimation    and most probable paths. Finally, we give some discussions and comments in a more biological context.

\section{Preliminaries}
We recall some basic definitions for L\'evy motions (or L\'evy processes).
\begin{definition}
A stochastic process $L_t$ is a  L\'evy  process if
\begin{enumerate}
\item[(1)]$L_0=0$  (a.s.);
\item[(2)]$L_t$ has independent increments and stationary increments; and
\item[(3)] $L_t$ has stochastically continuous sample paths, i.e.,  for  every  $s\geq 0$,
 $L(t)\rightarrow L(s)$ in probability, as $t \rightarrow s $.
\end{enumerate}
\end{definition}
A L\'evy process $L_t$ taking values in $\mathbb{R}^n$ is characterized by a drift vector $b \in {\mathbb{R}^n}$, an $n \times n$ non-negative-definite, symmetric covariance matrix $ Q $ and a Borel measure $\nu$ defined on ${\mathbb{R}^n}\backslash \{ 0\} $.   We call $(b,Q,\nu)$ the generating triplet of  the L\'evy motions $L_t$. Moreover, we have the L\'evy-It\^o decomposition for $L_t$ as follows:
\begin{equation}
{L_t}= bt + B_{Q}(t) + \int_{||y||< 1} y \widetilde N(t,dy) + \int_{||y||\ge 1} y N(t,dy),
\end{equation}
where $N(dt,dy)$ is the Poisson random measure, $\widetilde N(dt,dy) = N(dt,dy) - \nu (dx)dt$ is the compensated Poisson random measure, $\nu (S) = \mathbb{E}N(1,S)$ is the jump measure, and $B_t$ is an independent standard $n$-dimensional Brownian motion. The characteristic function of $L_t$ is given by
\begin{equation}
\mathbb{E}[\exp({\rm i}\langle u, L_t \rangle)]=\exp(t\psi(u)), ~~~u \in {\mathbb{R}^n},
\end{equation}
where the function $\psi:{\mathbb{R}^n}\rightarrow \mathbb{C}$ is  the characteristic exponent
\begin{equation}
\psi(u)={\rm i}\langle u, b\rangle-\frac{1}{2}\langle u, Qu \rangle+\int_{{\mathbb{R}^n}\backslash \{ 0\}}{(e^{{\rm i}\langle u, z \rangle}-1-{\rm i}\langle u,z\rangle {I_{\{ || z || \textless 1\} }})\nu(dz)}.
\end{equation}
The Borel measure $\nu $ is called the jump measure. Here $\langle \cdot, \cdot \rangle $ denotes the scalar product in $ \mathbb{R}^n $.
\begin{definition}
For $\alpha \in (0,2)$, an  $ n $-dimensional symmetric $\alpha $-stable process $ L^{\alpha}_{t} $ is a L\'evy process with characteristic exponent $\psi$
\begin{equation}
\psi(u)=-C_1(n,\alpha)| u |^{\alpha},  ~for~u \in {\mathbb{R}^{n}}
\end{equation}
with $ C_1(n, \alpha):=\pi^{-\frac{1}{2}}\Gamma((1+\alpha)/2)\Gamma(n/2)/\Gamma((n+\alpha)/2)$.
\end{definition}

For an $n$-dimensional symmetric $\alpha$-stable L\'evy process, the diffusion matrix $ Q= 0$,
the drift vector $ b= 0$, and the L\'evy measure $\nu $ is given by
\begin{equation}
\nu(du)=\frac{C_2(n,\alpha)}{{|| u ||}^{n+\alpha}}du,
\end{equation}
where $ C_2(n, \alpha):=\alpha\Gamma((n+\alpha)/2)/{(2^{1-\alpha}\pi^{n/2}\Gamma(1-\alpha/2))}$.

For every function $\phi\in \mathcal{C}^2_{0}({\mathbb{R}^n})$, the generator for this   symmetric $\alpha $-stable L\'evy process in $\mathbb{R}^n$ is
\begin{equation}
(\mathcal{A}_{\alpha}\phi)(x)= \int_{{\mathbb{R}^n}\backslash \{ 0\}}[\phi(x+u)-\phi(x)] \; \nu(du).
\end{equation}
It is known \cite{da} that $\mathcal{A}_{\alpha}$ extends uniquely to a self-adjoint operator in the domain. By Fourier inverse transform,
\begin{equation}
\mathcal{A}_{\alpha}\phi=\theta_{\alpha, n}(-\Delta)^{\alpha/2}\phi(x),
\end{equation}
where
\begin{equation}
\theta_{\alpha, n}=\int_{{\mathbb{R}^n}\backslash \{ 0\}}(\cos(\langle
e, y\rangle)-1) \; \nu(du)
\end{equation}
with $ e $ being the unit vector in $ \mathbb{R}^n $.

For fix $\alpha \in (1,2), \eta \in [0, \alpha)$, and set
\begin{equation}
\rho(x)=\big(1+| x |\big)^{\eta}.
\end{equation}
Let $L^{p}_{\rho}(\mathbb{R}^{n})$ be the weighted $ L^{p}$-space with norm:
\begin{equation}
|| f ||_{p;\rho}=\big(\int_{\mathbb{R}^{n}} | f(x)|^{p}\rho(x)dx\big)^{\frac{1}{p}}
\end{equation}
For $l\in {\mathbb{N}}$, let $W_{\rho}^{l,p}$ be the $l$-order weighted Sobolev space with norm
\begin{equation}
|| f ||_{l,p;\rho}=\sum_{i=0}^l || \nabla ^{l} f ||_{p;\rho},
\end{equation}
where $\nabla^{l}$ denotes the $l$-order gradient.
\begin{definition}
A backward predictable stochastic process $v \in L^{p}(\Omega; L_{loc}(\mathbb{R};W^{1,p}_{\rho}))$ is called a generalized solution of the equality
\begin{equation}
v(t,x)=\varphi(x)+\int_{[t,T]}Lv(s,x)ds +\int_{[t,T]}h^{l}v(s,x)dw^{l}(s),
\end{equation}
if for every $y \in C_0^{\infty}(\mathbb{R}^{d})$ it satisfies the following equation
\begin{equation}
(v(t,x),y)=(\varphi,y)+\int_{[t,T]}(v(s,x),L^{*}y)ds
+\int_{[t,T]}(h^{l}v(s,x),y)dw^{l}(s)
\end{equation}
\end{definition}
with $L$ being the generator of some L\'evy process.
\section{A slow-fast filtering  system}

Let us consider the following slow-fast signal-observation system:
\begin{numcases}{}
  d{X_t} = {f_1}({X_t},{Y_t})dt + {g_1}({X_t},{Y_t})d{B^1_t} +\sigma_1 dL_1^{{\alpha _1}},  \;\;  x\in \mathbb{R}^n, \\
  d{Y_t} = {\varepsilon ^{ - 1}}{f_2}({X_t},{Y_t})dt + {\varepsilon ^{ - \frac{1}{2}}}{g_2}({X_t},{Y_t})dB_t^2 + \frac{1}{\varepsilon ^{\frac{1}{\alpha _2}}}dL_2^{{\alpha _2}},  y \in \mathbb{R}^m, \\
  d{Z_t} = h({X_t})dt + d{W_t}, \; z\in \mathbb{R}^d.
  \end{numcases}
Here $({X_t},{Y_t})$ is an ${\mathbb{R}^n} \times {\mathbb{R}^m}$ -valued signal process which represents the slow and fast components. The constant $\sigma_1$ represents the noise intensity for the slow variable. The observation process ${Z_t}$ is ${\mathbb{R}^d}$ -valued.   The standard Brownian motions $B^1_t,B^2_t, W_t$ are   independent.  The non-Gaussian   processes  $L_1^{{\alpha _1}},L_2^{{\alpha _2}}$  (with $1<\alpha_1, \alpha_2 <2$) are independent symmetric $\alpha$-stable L\'evy processes with triplets $(0,0,{\nu_1})$ and $(0,0,{\nu _2})$, respectively.   The parameter $\varepsilon$ is the ratio of the slow time scale to the fast time scale.
We make the following assumptions  on this filtering system.\\
{\bf Hypothesis  H.1. }
The functions $f_1, g_1, f_2, g_2$ satisfy the global Lipschitz conditions, i.e., there exists a positive constant $\gamma$ such that
$$\begin{array}{l}
|| f_1(x_1,y_1)-f_1(x_2,y_2)||^{2}+|| g_1(x_1,y_1)-g_1(x_2,y_2)||^{2}\\+|| f_2(x_1,y_1)-f_2(x_2,y_2)||^{2}+|| g_2(x_1,y_1)-g_2(x_2,y_2)||^{2}\\\leq \gamma [ || x_1-x_2||^2+|| y_1-y_2||^2 ]\end{array}$$

for all $x_i\in \mathbb{R}^{n},y_i\in \mathbb{R}^{m}, \; i=1, 2.$
\begin{remark}
Note that with the help of the global Lipschitz condition, it   follows that there is a positive constant $ K $ such that
\begin{equation*}
|| f_1(x,y)||^2+|| g_1(x,y)||^{2}+|| f_2(x,y)||^2+|| g_2(x,y)||^{2}\leq K(1+|| x||^2 +|| y||^2 )
\end{equation*}
for all $x\in \mathbb{R}^{n},y\in \mathbb{R}^{m}$.
\end{remark}
\par
{\bf Hypothesis  H.2. }
The coefficients $ f_2, g_2 $ are  of class $ C^{2}(\mathbb{R}^{n}\times \mathbb{R}^{m})$  with the first and second order derivatives bounded

\par
{\bf Hypothesis  H.3. }
 The sensor function $ h $ is  of class $ C_{b}(\mathbb{R}^{n})$, i.e. all the bounded continuous functions on $\mathbb{R}^{n}$  .

\par
{\bf Hypothesis  H.4. }
There exists a positive constant $C$, such that
\begin{equation}
\sup\limits_{x}\{\mbox {Tr}[g_2(x,y)g_2(x,y)^{T}]+2\langle y,f_2(x,y)\rangle \}\leq -C(1+| y |^2).
\end{equation}

This hypothesis (H.4.) ensures the existence of an invariant measure  $\mu_{x}(dy)$ (see \cite{wu}) for the fast component $Y_t$. With the special scaling exponent  $ {\frac{1}{\alpha _2}}$ for the fast component $Y_t$, this invariant measure is independent of $\varepsilon$ (see \cite{ar,wu}).

The infinitesimal generator of the slow-fast stochastic  system  is
\begin{equation}
 \mathcal{L}= {\mathcal{L}_1} + \frac{1}{\varepsilon }{\mathcal{L}_2},
\end{equation}
where
\begin{equation}
\left\{
\begin{array}{rl}
{\mathcal{L}_1}\Phi(x,y,t)=& {f_1}\cdot {\nabla_x} \Phi(x,y,t)+\frac{1}{2}\mbox{Tr}[{G_1}\cdot {H_x}(\Phi(x,y,t) )] \\
&+\mbox{P.V.}\int_{{\mathbb{R}^n}\backslash \{ 0\}} {(\Phi (x + \sigma_1z_1,y,t) - \Phi (x,y,t))}{\nu _1}(dz_1), \\
 {\mathcal{L}_2}\Phi (x,y,t) = &{f_2}\cdot {\nabla_y }\Phi(x,y,t)+ \frac{1}{2}\mbox{Tr}[ {G_2}\cdot {H_y}(\Phi(x,y,t) )] \\
&+\mbox{P.V.}\int_{{\mathbb{R}^m}\backslash \{ 0\}} {(\Phi (x,y + z_2,t)-\Phi (x,y,t))}{\nu _2}(dz_2).
\end{array}
\right.
\end{equation}
Here $G_i=g_ig_i^T\,(i=1,2)$,  $\nabla$ is the gradient,
and $H $ is the Hessian matrix (with respect to $x$ and $y$ respectively).
Let
\begin{equation}
\mathcal{Z}_t=\sigma(Z_s:0\leq s\leq t)\vee \mathcal{N},
\end{equation}
where $\mathcal{N}$ is the  collection of all $\mathbb{P}$ -negligible sets of $(\Omega,\mathcal{F})$.
Define
\begin{equation}
\mathcal{Z}=\bigvee_{t\in {\mathbb{R}}^{+}}\mathcal{Z}_t,
\end{equation}
where $\bigvee $  denotes taking the $\sigma $-algebra generated by the union $\bigcup_t\mathcal{Z}_t $. That is,
\begin{equation}
\mathcal{Z}=\sigma(\bigcup_t\mathcal{Z}_t ).
\end{equation}

By the version of Girsanov's change of measure theorem, we obtain a new probability measure $\tilde \mathbb{P}$, such that the
observation $Z_t$ becomes $\tilde \mathbb{P}$-independent of the signal variables $(X_t,Y_t)$. This can be done through
\begin{equation}
\frac{{d\tilde \mathbb{P}}}{{d\mathbb{P}}}= \exp\left(-\sum\limits_{i = 1}^m {\int_0^t {{h^i}({X_s})dW_s^i} }- \frac{1}{2}\sum\limits_{i = 1}^m {\int_0^t {{h^i}{{({X_s})}^2}} } ds\right).
\end{equation}
Denote
\begin{equation}
\tilde{R}_t=\left.\frac{{d\mathbb{P}}}{{d\tilde\mathbb{P}}}\right|_{\mathcal{Z}_t}.
\end{equation}
Then by the Kallianpur-Striebel formula,
for every bounded differentiable function $\phi$, we have the following representation:
\begin{equation}
 \mathbb{E}[\phi(X_t,Y_t)|\mathcal{Z}_t]=\frac{\tilde\mathbb{{E}}[\tilde {{R_t}}\phi({X_t},{Y_t})|\mathcal{Z}_t]}{\tilde\mathbb{{E}}[\tilde {{R_t}}|\mathcal{Z}_t]}=\frac{\tilde\mathbb{{E}}[\tilde {{R_t}}\phi({X_t},{Y_t})|\mathcal{Z}]}{\tilde\mathbb{{E}}[\tilde {{R_t}}|\mathcal{Z}]}.
\end{equation}
 The unnormalized conditional distribution of $\phi(X_t,Y_t)$, given $Z_t$, is defined as ${P_t}(\phi) = \tilde\mathbb{E}[\tilde {{R_t}}\phi({X_t},{Y_t})|\mathcal{Z}]$. Thus, we have the following Zakai equation:
\begin{equation}
\left\{
\begin{array}{rl}
d{P_t}(\phi)=& P_t(\mathcal{L}\phi)dt + {P _t}(\phi{h^T})d{Z_t},\\
{P _0}(\phi)=&\mathbb{E}[\phi(X_0,Y_0)].\\
\end{array}
\right.
\end{equation}
The $x$-marginal of ${P_t}(\phi)$ is defined as
\begin{equation}
\rho_{t}(\phi)=\int_{\mathbb{R}^{n+m}}\phi(x)P_t(dx,dy).
\end{equation}
Now we define a reduced, low dimensional signal-observation system
\begin{equation}
\label{6.3}
\left\{
\begin{array}{rl}
 d\bar{X}_{t}=&\bar {{f_1}}(\bar{X}_{t})dt + \bar{{g_1}}(\bar{X}_{t})dB_t^1 +\sigma_1 dL_1^{{\alpha _1}}, \\
 d\bar{Z}_{t}=& h(\bar{X}_{t})dt + d{B_t}.\\
\end{array}
\right.
\end{equation}
Here
\begin{equation}
\left\{
\begin{array}{rl}
\bar{{f_1}} (x) =& \int_{\mathbb {R}^{m}} {{f_1} (x,y){\mu _x}(dy)},\\
 \bar{G}(x)=&\int_{\mathbb{R}^{m}} {{g_1}(x,y){g_1}{{(x,y)}^T}{\mu _x}(dy)}.
\end{array}
\right.
\end{equation}
The unnormalized conditional distribution $\bar\rho_{t}(\phi)$ corresponding to the filter for the reduced system (\ref{6.3}) satisfies
the following (reduced) Zakai equation
\begin{equation}
\label{6.4}
\left\{
\begin{array}{rl}
d\bar\rho_{t}(\phi)= & \bar\rho_{t}(\bar\mathcal{{L}}\phi) dt + \bar \rho_{t}(\phi h^T)d\bar{Z_t}, \\
\bar\rho_{0}(\phi)=&\mathbb{E}[{\phi(X_0)}],
\end{array}
\right.
\end{equation}
where
\begin{equation}
\begin{array}{rl}
\bar\mathcal{{L}}\phi (x,t)= &\bar{f} _1(x)\cdot \nabla _x\phi(x,t)+ \frac{1}{2}\mbox{Tr}[{H_x}(\bar{G}(x)\phi(x,t))]\\
+& \int_{\mathbb{R}^n} {(\phi  (x + \sigma_1z_1,t) - \phi (x,t)){\nu _1}(dz_1)}.
\end{array}
\end{equation}
However, we are more interested in the reduced filtering problem with  the actual observation $Z_t$.
This leads us to rewrite the reduced Zakai equation (\ref{6.4}) as follows:
\begin{equation}
\label{6.5}
\left\{
\begin{array}{rl}
d\bar\rho_{t}(\phi)= & \bar\rho_{t}(\bar\mathcal{{L}}\phi) dt + \bar \rho_{t}(\phi h^T)d{Z_t}, \\
\bar\rho_{0}(\phi)=&\mathbb{E}[{\phi(X_0)}].
\end{array}
\right.
\end{equation}

\begin{lemma}
Assume that the following conditions are satisfied:
\begin{enumerate}
\item[(i)] The functions $\sigma^{il}, b^{i}$, and $ h^{l}$ for $i,j=1,2,\cdots, k $ and $ l=1,2,\cdots, k_1 $ and their derivatives of first order (in x) as well as the derivatives of second order (in x) are uniformly bounded by the constant $ N_0 $. The functions $\sigma^{il}, b^{i}$ are locally uniformly bounded in $ t $.
\item[(ii)] $\phi \in \mathcal{W}_{\rho}^{1,p}(\mathbb{R}^{k})$.\end{enumerate}
Let $u(t,x,\omega)$ be a generalized solution of problem
\begin{equation}
\label{uw}
\left\{
\begin{array}{rl}
du(t,x,\omega)&=Lu(t,x,\omega)dt+h^{l}(t,x,\omega)u(t,x,\omega)dB^{l}(t),\\
&(t,x,\omega) \in [T_0, T]\times \mathbb{R}^{k}\times \Omega,\\
u(T_0,x,\omega)&=\varphi(x,\omega),  (x,\omega)\in \mathbb{R}^{k}\times \Omega,
\end{array}
\right.
\end{equation}
where
\begin{equation}
\begin{array}{rl}
Lu(t,x)&=\frac{1}{2}\sigma^{ik}(t,x, \omega)\sigma^{jk}(t,x, \omega)\partial_{ij}^2 u(t,x,\omega)+b^{i}\partial_iu(t,x,\omega)\\
&+\frac{1}{2}\int_{\mathbb{R}^{k}}(u(t,x+z)+u(t,x-z)-2u(t,x))\frac{dz}{| z |^{k+\alpha}},
\end{array}
\end{equation}
and $v(t,x,\omega)$ is a generalized solution of problem
\begin{equation}
\label{ov}
\left\{
\begin{array}{rl}
-dv(t,x,\omega)&=[Lv(t,x,\omega)]dt+[h^{l}(t,x,\omega)v(t,x,\omega)]dB^{l}(t),\\
&(t,x,\omega) \in [T_0, T]\times \mathbb{R}^{k}\times \Omega,\\
v(T,x,\omega)&=\varphi(x,\omega), (x,\omega)\in \mathbb{R}^{k}\times \Omega.
\end{array}
\right.
\end{equation}
Then the following formulas hold
\begin{equation}
\label{u}
u(t,x)=\mathbb{E}(\varphi(Y(t,x,T_0))\chi_1(t,T_0)|\mathcal{F}_t^{T_0}),
\end{equation}
and
\begin{equation}
\label{vv}
v(t,x)=\mathbb{E}(\varphi(X(T,x,t))\chi_2(T,t)|\mathcal{F}_t^{T}),
\end{equation}
where $ X(t,x,s) $ is the forward stochastic differential equation with generator $ L $ and
$Y(t,x,s)$ is the backward stochastic differential equation with generator $ L $. Here $\chi_1(t,s)$
and $\chi_2(t,s)$ satisfy the following equations
\begin{equation}
\chi_1(t,s)=exp\left(\int_{[s,t]}h^{l}(\tau, Y(t,x,\tau))dB^{l}(\tau)-\frac{1}{2}\int_{[s,t]}h^{l}h^{l}(\tau, Y(t,x,\tau))d\tau \right),
\end{equation}
and
\begin{equation}
\chi_2(s,t)=exp\left(\int_{[t,s]}h^{l}(\tau, X(\tau,x,t))dB^{l}(\tau)-\frac{1}{2}\int_{[t,s]}h^{l}h^{l}(\tau, X(\tau,x,t))d\tau\right).
\end{equation}
\end{lemma}
\begin{proof}
Denote
$ a=(a^{ij})=\frac{1}{2}\sigma(t,x, \omega)\sigma^{T}(t,x, \omega) $. Before we proceed to the proof of the lemma, let us consider the problem
\begin{equation}
\left\{
\label{vvv}
\begin{array}{rl}
dv(t,x,\omega)&=[a^{ij}(T-t,x,\omega)\partial_{ij}^2 v(t,x,\omega))
+b^{i}(T-t,x,\omega)\partial_i v(t,x,\omega)\\
&+\frac{1}{2}\int_{\mathbb{R}^{k}}(v(t,x+z)+v(t,x-z)-2v(t,x))\frac{dz}{| z |^{k+\alpha}}]dt\\
&+h^{l}(T-t,x,\omega)v(t,x,\omega)dB_{T}^{l}(t),(t,x,\omega) \in [0, T-T_0]\times \mathbb{R}^{k}\times \Omega,\\
v(0,x,\omega)&=\varphi(x,\omega), (x,\omega)\in \mathbb{R}^{k}\times \Omega.
\end{array}
\right.
\end{equation}

By the definition of a backward predictable stochastic process, the coefficients in equation (\ref{vvv}) are predictable relative to the family of $\sigma-$ algebra ${\mathcal{F}_{T}^{T-t}}$ with $ t \in [0,T-T_0]$. The process $ B_{T}$ is a Wiener martingale
with respect to the same family and the initial condition $\varphi$ is measurable with respect to
the minimal $\sigma$-algebra of this family.

Let $v(t,x)$ be a generalized solution of problem. Then by the definition of generalized
solution, for every $ y \in C_0^{\infty}(\mathbb{R}^{k})$, the following equality holds on $[T_0, T]\times \Omega$,
\begin{equation}
\label{change}
\begin{array}{rl}
(v(T-s,x),y)&=(\varphi,y)+\int_{[0,T-s]}\left[-(a^{ij}(T-s)\partial_iv(s,x),y_j)\right.\\
&+(b^{i}(T-s)-a^{ij}_{j}(T-s))\partial_iv(s,x))\\
&\left.+\frac{1}{2}\int_{\mathbb{R}^{k}}(v(s,x+z)+v(s,x-z)-2v(s,x))\frac{dz}{| z |^{k+\alpha}},y)\right]ds \\
&+\int_{[0,T-s]}h^{l}(T-t)v(t,x)dB^{l}_{T}(s).
\end{array}
\end{equation}
Denote $u(s)=v(T-s)$ in equation (\ref{change}),
we find that $u(t)$ satisfies the following equation on $[T_0, T]\times \Omega $, for every $ y \in C_0^{\infty}(\mathbb{R}^{k})$.
\begin{equation}
\label{chdd}
\begin{array}{rl}
(u(t,x),y)&=(\varphi,y)+\int_{[t,T]}\bigl[-(a^{ij}\partial_iv(s,x),y_j)
+((b^{i}-a^{ij}_{j})\partial_iu(t,x))\\
&+\frac{1}{2}\int_{\mathbb{R}^{k}}(v(t,x+z)+v(t,x-z)-2v(t,x))\frac{dz}{| z |^{k+\alpha}},y)\bigr]ds \\
&+\int_{[t,T]}h^{l}v(s,x)dB^{l}(s).
\end{array}
\end{equation}
Thus $ u $ is an generalized solution of problem (\ref{ov}).

On the other hand, changing the variables in equality (\ref{chdd}), we obtain that $v(t,x)$ is a
generalized solution of problem (\ref{vvv}). Thus we have proved that problems (\ref{ov}) and (\ref{vvv}) are equivalent. Therefore all the results obtained for problem (\ref{ov}) are naturally carried over to problem (\ref{vvv}). For the problem (\ref{vvv}), we can use the similar method to obtain the existence and uniqueness of $\mathbb{W}^{1,p}_{\rho}$- solution to stochastic fractal equations by using purely probabilistic argument (see \cite {blr, xzz}).
This completes the proof of this lemma.
\end{proof}
Now we show that the reduced  system  approximates the original system, by examining the corresponding Zakai equations. Before describing the theorem, we start by describing the probabilistic representation for semi-linear stochastic fractal equations. We proceed to consider the probability measure $\mathbb{\widetilde P}$.  Note that the process $Z_t$ is a Brownian motion under $\mathbb{\widetilde P}$.  For convenience, we rewrite $\mathbb{P}$ and  $Z_t$  as $\mathbb{\widetilde P}$ and $B_t$, respectively. By Lemma 1, we know   $v^{\phi}(t,x,y)=\tilde\mathbb{{E}}_{x,y}[\phi(X_{{T}})\tilde{R}_{t,T}|\mathcal{Y}_{t,T}]$ solves a stochastic partial differential equation (SPDE).
\begin{equation}
\label{oriback}
\left\{
\begin{array}{rl}
 -dv^{\phi}(t,x,y)=&\mathcal{L}v^{\phi}(t,x,y)dt +h(x)^{T}v^{\phi}(t,x,y)d\overleftarrow{B_t}. \\
 v^{\phi}(T,x,y) =&\phi(x).
\end{array}
\right.
\end{equation}
Here $d\overleftarrow{B_t}$ denotes It\^o's backward integral. Likewise, we introduce $\overline v^{\phi}(t,x)=\tilde\mathbb{{E}}_{x}[\phi(\bar{X}_{T})\tilde{R}_{t,T}|\mathcal{Y}_{t,T}]$, and then $\overline v(t,x)$ solves the following SPDE
\begin{equation}
\label{reduback}
\left\{
\begin{array}{rl}
 -d\overline v^{\phi}(t,x)=&\overline{\mathcal{L}}\overline v^{\phi}(t,x)dt + h(x)^{T}\overline v^{\phi}(t,x)d\overleftarrow{B_t}, \\
 \overline v^{\phi}({ T},x)=&\phi(x).
\end{array}
\right.
\end{equation}
By the version of Girsanov's change of measure theorem and Markov property of $(X_t, Y_t)$, we know that for any $t\in [0,T]$: $P_t(v^{\phi}(t))=\rho_{T}(\phi)$. In particular, we have
\begin{equation}
\rho_{T}(\phi)=\int v^{\phi}(0,x,y)\mathbb{P}_{(X_0,Y_0)}(dx,dy).
\end{equation}
Similarly, we have
\begin{equation}
\bar{\rho}_{T}(\phi)=\int \bar{v}^{\phi}(0,x)\mathbb{P}_{(X_0)}(dx).
\end{equation}
This is our main result on the comparison between the original filter and the reduced filter. This is desirable  when only the slow component $X_t$ is observable.
\begin{theorem}
Under the hypotheses (H.1)-(H.4) and for $p$ with $1 \textless p  \textless \alpha_1$, there exists a positive constant $ \varepsilon_{0}$ such that for $ \varepsilon \in (0, \varepsilon_{0})$, the following estimate holds
\begin{equation}\label{result}
\mathbb{E}[|\rho_{T}(\phi)- \bar{\rho}_{T}(\phi)|^{p}]\leq C \varepsilon,
\end{equation}
with a positive constant $C$ independent of $ \varepsilon_{0}$.
This implies that the reduced filter approximates the original filter, as the scale parameter $\varepsilon$ tends to zero.
\end{theorem}
\begin{proof}\hspace{2cm}
\begin{equation}
\begin{array}{rl}
&\mathbb{E}[|\rho_{T}(\phi)- \bar{\rho}_{T}(\phi)|^{p}]\\
=&\mathbb{E}[\left | \int (v^{\phi}(0,x,y)-\bar{v}^{\phi}(0,x))\mathbb{P}_{(X_0,Y_0)}(dx,dy)\right|^{p} ], \\
\leq& \mathbb{E}[\int\left\| v^{\phi}(0,x,y)-\bar{v}^{\phi}(0,x)\right|^{p}\mathbb{P}_{(X_0,Y_0)}(dx,dy)],\\
=&\int\mathbb{E}\left[\left| v^{\phi}(0,x,y)-\bar{v}^{\phi}(0,x)\right|^{p}\right]\mathbb{P}_{(X_0,Y_0)}(dx,dy),\\
=&\int\mathbb{E}\left[\left|\tilde\mathbb{{E}}_{x,y}[\phi(X_{{T}})\tilde{R}_{0,T}|\mathcal{Y}_{0,T}]-
\tilde\mathbb{{E}}_{x}[\phi(\bar{X}_{{T}})\tilde{R}_{0,T}|\mathcal{Y}_{0,T}]\right|^{p}\right]\mathbb{P}_{(X_0,Y_0)}(dx,dy),\\
\leq& C \int\mathbb{E}\left|\tilde\mathbb{{E}}_{x}[(\phi(X_{{T}})-\phi(\bar{X}_{{T}}))\tilde{R}_{0,T}|\mathcal{Y}_{t,T}]\right|^{p}
\mathbb{P}_{(X_0,Y_0)}(dx,dy),\\
=& C \int\mathbb{E}\left|\mathbb{E}[\phi(X_{T})-\phi(\bar{X}_{T})|\mathcal{Y}_{T}]
\tilde\mathbb{{E}}_{x}[\tilde{R}_{0,T}|\mathcal{Y}_{0,T}]\right|^{p}\mathbb{P}_{(X_0,Y_0)}(dx,dy),\\
\leq&  C \int \mathbb{E} \left|\mathbb{E}[X_{T} -\bar{X}_{T}|\mathcal{Y}_{T}]\right |^{p}\mathbb{P}_{(X_0,Y_0)}(dx,dy).
\end{array}
\end{equation}
Using the Jensen's inequality, we have
\begin{equation}
\mathbb{E}[|\rho_{T}(\phi)- \bar{\rho}_{T}(\phi)|^{p}]\leq \int \mathbb{E}| X_{T} -\bar{X}_{T}|^{p}\mathbb{P}_{(X_0,Y_0)}(dx,dy).
\end{equation}
Applying a similar  argument from \cite{bg}, we obtain
\begin{equation}
\lim\limits_{\epsilon \to 0}\mathbb{E}| X_{T} -\bar{X}_{T}|^{p}=0
\end{equation}
Finally, we conclude that
\begin{equation}
\mathbb{E}[|\rho_{T}(\phi)- \bar{\rho}_{T}(\phi)|^{p}]\leq C\varepsilon.
\end{equation}
This completes the proof of Theorem 1.
\end{proof}
\section{Parameter estimation}

In this section, we consider parameter estimation in the following slow-fast dynamical system
\begin{equation}
\label{ori}
\left\{
\begin{array}{rl}
& dx(t)= \{Ax(t)+{f_1}(x(t),y(t),\theta)\}dt, \; x \in  \mathbb{R}^n, \\
& dy(t) = \frac{1}{\ep}\{By(t)+{f_2}(x(t),y(t))\}dt + \frac{1}{\ep ^{\frac{1}{\alpha }}}dL_t^{{\alpha}}, \;   y \in \mathbb{R}^m,
\end{array}
\right.
\end{equation}
where $A$ is an $ n\times n $ matrix, $-B$ is a positive definite matrix with eigenvalues $0 < \lambda_1 < \lambda_2 <\cdots < \lambda_m $. $\theta $ is an unknown parameter defined in an compact set $\Theta$ of $ \mathbb{R}^n $.
\par
{\bf Hypothesis  H.5. }
For all $\theta\in\Theta $, the function $f_1$ is uniformly bounded and Lipschitz w.r.t x and y. The function $f_2 $ is smooth with the first order derivatives bounded by $\lambda_1$.

Define an averaged, low dimensional slow stochastic dynamical  system in $\mathbb{R}^n$ (see \cite{andrew}), as in the previous section
\begin{equation}
\label{red}
d\bar{x}(t)=A\bar{x}(t)+\bar {{f_1}}(\bar{x}(t),\theta)dt,
\end{equation}
where
\begin{equation*}
\bar{{f_1}} (\bar{x},\theta)   \triangleq   \int_{\Rm} {{f_1} (x,y,\theta){\mu _x}(dy)}.
\end{equation*}
We will estimate the unknown parameter $\theta$, based on this low dimensional slow system   (\ref{red}) but using the observations on the slow component $x$ only. Denote the solution of the original slow-fast system (\ref{ori}) with actual parameter $\theta_0$ by $(x(t,\theta_0),y(t,\theta_0))$,  and the solution of the  slow system  (\ref{red}) with parameter $\theta$ by $\bar x(t,\theta)$. We recall the following lemma.
\begin{lemma}
Under hypothesis (H.6),  the following strong convergence holds
\begin{equation}
\lim_{\ep\to 0}\E|x(t,\theta_0)-\bar x(t,\theta_0)|^p=0, ~t\in[0,T]~and~p\in(1,\alpha).
\end{equation}
\end{lemma}
\begin{proof}
The proof of average principle is similar the the infinite case which has been studied in \cite{bg}, Thus we omit here.
\end{proof}
We take $F(\theta)=\int_0^T | x(t,\theta_0)-\bar x(t,\theta)|^p dt$ as the objective function and assume there is a unique $ \hte^{\epsilon}\in\Theta $ such that $F(\hte^{\epsilon})=\min\limits_{\theta\in\Theta}
F(\theta)$. This $\hte^{\epsilon}$ is our estimated parameter value. Then we can state our main result as follows.
\begin{theorem}
Under hypothesis (H.6), the estimated parameter value converges to the true parameter value, as the scale parameter $\ep$ approaches zero. That is,
\begin{equation}\label{result}
\lim\limits_{\ep\to 0}\hte^{\epsilon}=\theta_0.
\end{equation}
\end{theorem}
\begin{proof}
Note that
\begin{equation}\label{eq0}
\E|\bar x(t,\theta_0)-\bar x(t,\hte^{\ep})|^p \le C\; (\E|x(t,\theta_0)-\bar x(t,\theta_0)|^p+\E|x(t,\theta_0)-\bar x(t,\hte^{\ep})|^p),
\end{equation}
for some positive constant $C$. Integrating both sides with respect to time, we get
\begin{equation}\label{eq1}
\int_0^T \E|\bar x(t,\theta_0)-\bar x(t,\hte^{\ep})|^p dt\le C(F(\theta_0)+F(\hte^{\ep}))\le 2CF(\theta_0).
\end{equation}
We calculate the difference between $\bar x(t,\theta_0)$ and $ \bar{x}(t,\hte^{\ep})$ to obtain
\begin{equation*}
\dot{\bar x}(t,\theta_0)-\dot{\bar{x}}(t,\hte^{\ep})=A[\bar x(t,\theta_0)-\bar{x}(t,\hte^{\ep})]+[\bar f_1(\bar{x}(t,\theta_0),\theta_0)-\bar {{f_1}}(\bar{x}(t,\hte^{\ep}),{\hte^{\ep}})].
\end{equation*}
By the variation of constant formula, we have
\begin{equation*}
\bar x(t,\theta_0)-\bar{x}(t,\hte^{\ep})=\int_{0}^{t}e^{A(t-s)}[\bar f_1(\bar x(s,\theta_0),\theta_0) - \bar{f_1}(\bar{x}(s,\hte^{\ep}),{\hte^{\ep}})]ds.
\end{equation*}
Using the mean value theorem, we obtain
\begin{equation}\label{eq2}
\bar x(t,\theta_0)-\bar{x}(t,\hte^{\ep})=(\theta_0-\hte^{\ep})\int_0^t e^{A(t-s)}\nabla_\theta\bar f_1(\bar x(s,\theta'), \theta')ds,
\end{equation}
for some $\theta'=\hte^{\ep}+l(\theta_0-\hte^{\ep})$ with $l\in(0,1)$.
Denote
\begin{equation*}
G(\theta')=\int_0^T\left|\int_0^t e^{A(t-s)}\nabla_\theta\bar f_1(\bar x(s,\theta'),\theta')ds\right|^p dt.
\end{equation*}
If $G(\theta')>0$, we have by (\ref{eq1}) and (\ref{eq2})
\begin{equation}
|\theta_0-{\hat{\theta}}^{\epsilon}|\le \frac{C}{G(\theta')}F(\theta_0),
\end{equation}
which implies the desired result (\ref{result}) due to Lemma 1.

If $G(\theta')=0$, we have $\int_0^T \E|\bar x(t,\theta_0)-\bar{x}(t,\hte^{\ep})|^p dt=0$. Then $F(\theta_0)=F(\hte^{\ep})$, which implies $\theta_0=\hte^{\ep}$ since $F(\theta)$ admits a unique $\hte^{\ep}$ such that $F(\hte^{\ep})=\min\limits_{\theta\in\Theta}F(\theta)$. This completes the proof of this theorem.
\end{proof}

We proceed to the following two dimensional slow-fast stochastic dynamical system in $\mathbb{R}^{2}$ to verify the parameter estimation method.
\begin{example}Consider the following slow-fast stochastic dynamics
\begin{equation}
\label{ex}
\left\{
\begin{array}{rl}
dx(t)&=-x(t)dt+ cos(\theta x)e^{-y^2}dt,\\
dy(t)&=-\frac{y(t)}{\ep}dt+ \frac{1}{\ep^{\frac{1}{\alpha }}}dL_t^{{\alpha }}.
\end{array}
\right.
\end{equation}
Using a result in \cite{ar}, we find the   invariant measure   $\mu(dx)=\rho(x)dx$ with density
\begin{equation}
\rho(x)=\frac{1}{{2\pi}}\int_{\mathbb{R}}e^{ix\xi}e^{-\frac{1}{\alpha}|\xi |^{\alpha}}d\xi.
\end{equation}
Then the averaged, one dimensional, slow equation is
\begin{equation}
d\bar{x}(t)=-\bar{x}(t)+ a\cos(\theta\bar{x}(t))dt,
\end{equation}
 where
\begin{equation*}
a=\frac{1}{\sqrt{4\pi}}\int_{\mathbb{R}}e^{-\frac{1}{4}\xi^2-\frac{1}{\alpha}| \xi|^{\alpha}}d\xi.\\
\end{equation*}

In the following numerical simulations, we apply the Euler-Maruyama method to generate $M$ different slow paths $\{x^{i,j}:x^{i,j}=x^{i,j}(t_i,\theta_0),i=1,2,..,N;j=1,2,...,M\}$  as the observations. Then we minimize the objective function $F(\theta)=\sum_{j=1}^M\sum_{i=0}^{N}| x^{i,j} -\bar{x}^{i}(\theta)|^{p}$ by the stochastic Nelder-Mead method \cite{rj} to obtain the estimated parameter value $\hat{\theta}$.

The deterministic Nelder-Mead method is a geometric search method to find a minimizer of an objective function, which was originally devised for nonlinear and deterministic optimization. However, when noise dominates the true differences in the function values, the relative rank of function values in the simplex can be wrong due to the sampling error, leading the algorithm to the wrong direction. The stochastic Nelder-Mead method overcomes this shortage.
\begin{figure}[!htb]
\centering
\subfigure[The sample paths $x$ vs. $\bar x$]{
\label{Fig.sub.1}
\includegraphics[width=0.45\textwidth]{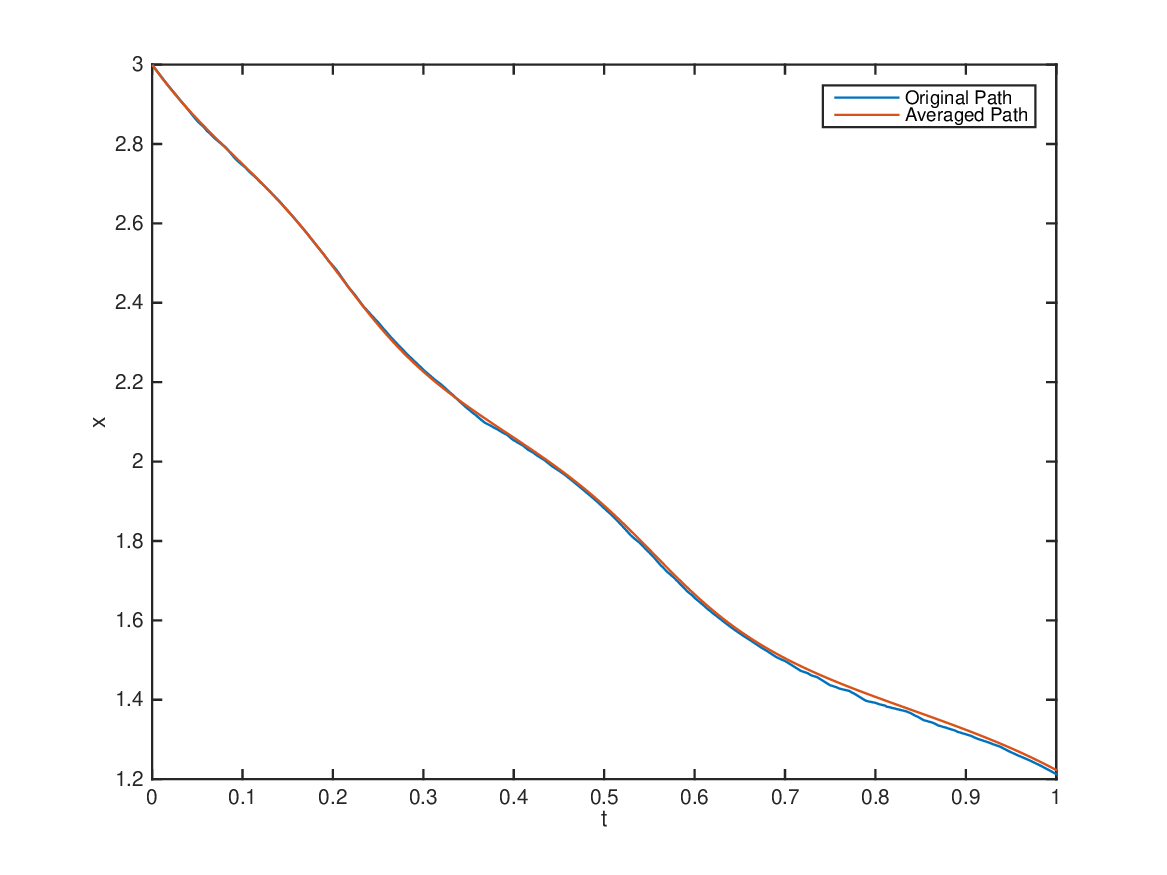}}
\subfigure[The estimated value $\hat\theta$]{
\label{Fig.sub.2}
\includegraphics[width=0.45\textwidth]{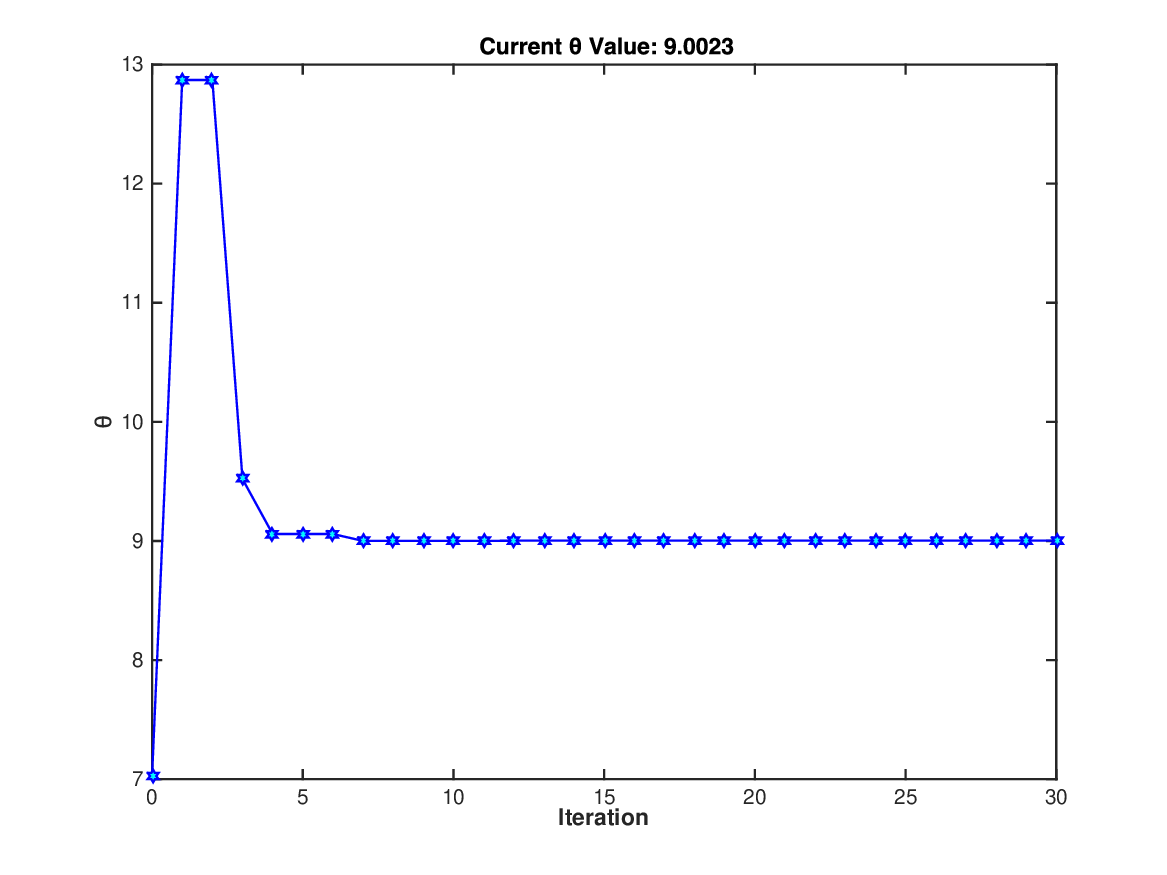}}
\caption{The sample paths $x, \bar x$ (online color)  and estimated value $\hat\theta$ (shown as $*$): $\theta=9 $ (true parameter value), $\ep=0.01,~\alpha=1.9,~\Theta=(0,100)$, $M=1000$. }\label{Fig 1}
\end{figure}

As shown in Figure 1, we see that slow system $\bar x$ is a good approximation of the slow component $x$ of the original system. By just about  $8$ iteration in the stochastic Nelder-Mead search, we get estimated parameter $\hat \theta \thickapprox 9$ (true parameter value).
\end{example}
\section{Most probable paths}
We further verify that  the average system captures the behaviour of the original system, by examing the most probable paths. We do this in a concrete example.

\begin{example}\hspace{2cm}
\begin{equation}
\left\{
\begin{array}{rl}
& d{X_t} = 0.1({X_t} - X_t^3)Y_t^2dt + {\left( {0.01} \right)^{\frac{1}{\alpha }}}dL_t^{{\alpha }},\\
& d{Y_t} =-\frac{Y_t}{\varepsilon }dt + \frac{2}{{\sqrt \varepsilon  }}d{B_t},\\
& d{Z_{1t}} ={X_t}dt+\sqrt{0.2}dW^{1}_t,\\
& d{Z_{2t}} ={Y_t}dt+\sqrt{0.2}dW^{2}_t,\\
\end{array}
\right.
\end{equation}


 The invariant measure $\rho ^\infty (y;x)$ for the fast variable $Y_t$ satisfies the normal distribution with mean 0 and variance 2 and does not depend on x. So we can get the averaged, one dimensional, slow equation with the original observation
\begin{equation}
\begin{array}{rl}
\bar{f} (x)&= \int_{\mathbb{R}} {0.1(x - {x^3}){y^2}{\rho ^\infty }(y;x)dy}\\
 &= 0.1(x - {x^3})\int_{\mathbb{R}} {{y^2}{\rho ^\infty }(y)dy}\\
 &= 0.2(x - {x^3}).
\end{array}
\end{equation}
and
\begin{equation}
\left\{
\begin{array}{rl}
& d{\bar{X}_t} = 0.2({\bar{X} _t} - \bar{X} _t^3)dt + {\left( {0.01} \right)^{\frac{1}{\alpha }}}dL_t^{{\alpha}}\\
& dZ_{1t}= X_tdt + \sqrt {0.2} dW_t
\end{array}
\right.
\end{equation}
 We can use partial differential equations methods to simulate the solution of the Zakai equation (see \cite{Crisan}), which can be carried out as the following two steps.
 \begin{figure}[!htb]
  \centering
  \includegraphics[width=2.5in,height=2.5in]{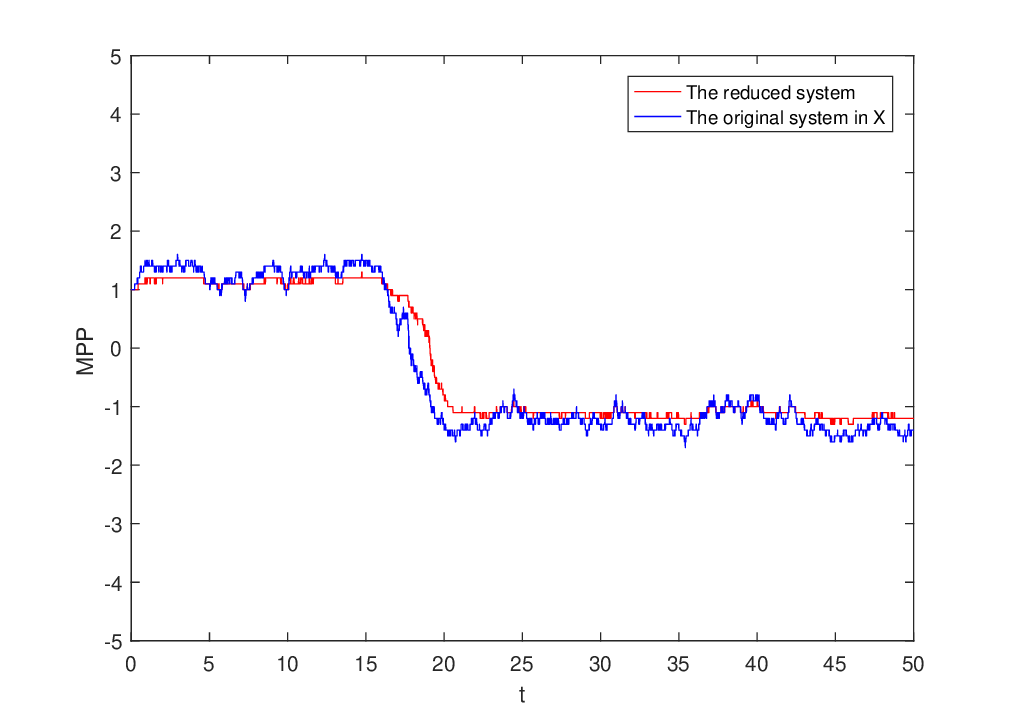}\\
  \caption{The most probable paths $X, \bar X$ (online color) with $X_0=1,\bar X_0=1,\ep=0.01,~\alpha=1.5.$ }\label{Fig 1}
\end{figure}

The first step, called the prediction step, consists of solving the Fokker-Planck equation(see \cite{Gao}).
The second step, called the correction step, uses the new observation to update.
Then we can infer the most probable phase portrait (see \cite{Cheng}) at each time t. Likely, we can  obtain the most probable path of the averaged system with the same observations. The detail illustration can be shown in the Figure 2.

\end{example}

\section{Conclusion and discussion}
We developed a reduction method based on a slow-fast stochastic dynamical system by using the stochastic averaging principle. We demonstrated that the low dimensional slow system approximates the slow dynamics of the original system, by examining the Zakai equation, parameter estimation and most probable paths. Rather than solving the original system, we can accurately estimate the unknown parameter $\theta $ only by the observation of the slow variable, which reduces the computational complexity  and will   not be computationally   expensive. The results established in this paper can be used to examine financial markets or stochastic chemical kinetics, where we are more interested in the change rate $\theta $ for low-risk bounds or mRNA.  Finally, we illustrated the low dimensional approximation by comparing the most probable paths for the slow system and the original system. This provided a deterministic geometric tool to visualize stochastic dynamics.

\medskip
\textbf{Acknowledgements}.  We would like to thank Peter Imkeller (Berlin, Germany) and Yanzhao Cao (Auburn, USA) for helpful discussions.


\end{document}